\documentclass[12pt,a4paper,reqno]{amsart}

\usepackage{amsmath}
\usepackage{amsfonts}
\usepackage{enumerate}

\newcommand{\Z}{\ensuremath{\mathbb{Z}}}
\newcommand{\R}{\ensuremath{\mathbb{R}}}
\newcommand{\C}{\ensuremath{\mathbb{C}}}

\DeclareMathOperator{\card}{Card}

\newcommand{\fracz}[2]{\genfrac{}{}{0pt}{}{#1}{#2}}
\newcommand{\eps}{\varepsilon}
\def\a{\mathrm{a}}
\def\b{\mathrm{b}}

\newtheorem{thm}{Theorem}[section]

\newtheorem{cor}[thm]{Corollary}
\newtheorem{lem}[thm]{Lemma}
\newtheorem{prop}[thm]{Proposition}

\newtheorem{rem}[thm]{Remark}

\title[Upper bounds for the number of zeroes]{Upper bounds for the number of
zeroes for some Abelian integrals}

\author[A. Gasull]{Armengol Gasull}
\address{Departament de Matem\`{a}tiques,
Universitat Aut\`{o}noma de Barcelona, 08193 Bellaterra, Barcelona,
Spain} \email{gasull@mat.uab.cat}

\author[J.~T. L\'azaro]{J.~Tom\'as L\'azaro}
\address{Departament de Matem\`atica Aplicada I, Universitat Polit\`ecnica de Catalunya, Barcelona, Spain}
\email{jose.tomas.lazaro@upc.edu}

\author[J.~Torregrosa]{Joan Torregrosa}
\address{Departament de Matem\`{a}tiques,
Universitat Aut\`{o}noma de Barcelona, 08193 Bellaterra, Barcelona,
Spain} \email{torre@mat.uab.cat}

\thanks{The first and third authors are partially supported by
the MICIIN/FEDER grant number MTM2008-03437 and by the Generalitat de
Catalunya grant number 2009SGR410. The second author is partially
 supported by the MICIIN/FEDER grant number MTM2009-06973 and by the
 Generalitat de Catalunya grant number 2009SGR859.}



\begin{document}

\begin{abstract}
Consider the vector field $x'= -yG(x, y), y'=xG(x, y),$ where the
set of critical points $\{G(x, y) = 0\}$ is formed by $K$ straight
lines, not passing through the origin and parallel to one or two
orthogonal directions. We  perturb it with a general polynomial
perturbation of degree $n$ and study which is the maximum number of
limit cycles that can bifurcate from the period annulus of the
origin in terms of $K$ and $n.$ Our approach is based on the
explicit computation of the Abelian integral that controls the
bifurcation and in a new result for bounding the number of zeroes of
a certain family of real functions. When we apply our results for
$K\le4$ we recover or improve some results obtained in several
previous works.

\end{abstract}

\maketitle

\section{Introduction} \label{se:1}

The problem of determining the number of limit cycles bifurcating
from the period annulus of a system
\begin{equation}\label{eq:1} 
\left\{
\begin{array}{rcl}
 \dot{x} &=& -yG(x,y) + \eps \,P(x,y), \\ [1.1ex]
 \dot{y} &=& \phantom{-} xG(x,y) + \eps \,Q(x,y), \\
\end{array}
\right.
\end{equation}
where $P(x,y),$ $Q(x,y)$ are polynomials of a given degree, $G(x,y)$
satisfies $G(0,0) \neq 0$ and $\eps$ is a small parameter, has been
widely studied (see for instance
\cite{BuiLli2007,ColGasPro2005,
GasProTor2008,GinLli2007,LiLliZha1995,LliPerRod2001,
XiaHan2004a,XiaHan2004b}).
Among this type of systems we will be concerned with those having
\begin{equation}\label{eq:2}
 G(x,y)=\prod_{j=1}^{K_1} (x-\a_j) \prod_{\ell=1}^{K_2}
 (y-\b_{\ell}),
\end{equation}
where $\a_j$ and $\b_\ell$ are real numbers with $\a_i\ne \a_j $ and
$\b_i\ne \b_j$ for $i\ne j.$ The unperturbed system ($\varepsilon=0$) presents 
a centre at the origin and any line $x=\a_j$ or
$y=\b_{\ell}$ constitutes an invariant set of singular points of the
system. This invariant set is formed by parallel and/or orthogonal
invariant lines.

The aim of this work is to provide, for small values of $\eps,$
upper bounds for the number of limit cycles bifurcating from
periodic orbits of the unperturbed system in the period annulus
\[
\mathcal{D}= \left\{(x,y)\in \R^2 \ \vert \ 0 < \sqrt{x^2+y^2} < \rho
:=\min_{j,{\ell}} \left\{ |\a_j|,|\b_{\ell}|\right\} \right\}.
\]
Several previous works handle this problem for different particular
choices of small values of $K_1$ and $K_2.$ The following cases have been
studied:  one line in~\cite{LliPerRod2001}; two parallel
lines in~\cite{XiaHan2004b}; two orthogonal lines
in~\cite{BuiLli2007}; three lines, two of them parallel and one
perpendicular in~\cite{ColLliPro2010}; and four lines with a special
configuration in~\cite{AtaNyaZan2009}. Other related works are
\cite{GinLli2007}, with $G(x,y)$ any quadratic polynomial;
\cite{XiaHan2004a}, one multiple singular line; or
\cite{GasProTor2008} for $K$ isolated singular points.

As it is standard in this type of problems, we rewrite our
system~\eqref{eq:1} in the equivalent form on
$\mathcal{D},$
\begin{equation} \label{eq:3}
\left\{
\begin{array}{rcr}
 \dot{x} &=& -y + \eps \, P(x,y)/G(x,y), \\
 \dot{y} &=&  x + \eps \, Q(x,y)/G(x,y),
\end{array}
\right.
\end{equation}
where $G(x,y)$ is defined in~\eqref{eq:2}. Let us denote
by $\gamma_r= \left\{ (x,y)\ \vert \ x^2+y^2=r^2\right\},$ with
$0<r<\rho,$ any periodic orbit of the unperturbed system. It is well
known that the associated \emph{return map}, whose isolated zeroes
give rise to limit cycles, is $\Pi(r,\eps)=r+I(r) \eps +
\mathcal{O}(\eps^2)$ where $I(r)$ is the Abelian integral
\begin{equation} \label{eq:4}
I(r)= \int_{\gamma_r} \dfrac{Q(x,y) \, dx - P(x,y) \, dy}{G(x,y)}.
\end{equation}
This integral is the so-called (first-order)
Poincar\'e-Melnikov-Pontryagin function. The return map
$\Pi(r,\eps)$ is analytic in $r$ ranging on any compact subset of
$(0,\rho)$ for $\eps$ small enough. It is known (see for instance~\cite{ChrLi2007}) that,
provided $I(r)$ does not vanish exactly, the number of zeroes of
$\Pi(r,\eps),$ for $\eps$ small enough, is at most the number of
zeroes of $I(r)$ taking into account their multiplicity. The problem of
estimating this number in terms of the involved degrees is commonly
called \emph{the weakened Hilbert's 16th Problem}.

The goal of this paper is to provide an upper bound for the number of
zeroes of the Abelian integral 
associated to a system of the form~\eqref{eq:1},
depending on the number of critical straight lines and the degree
$n$ of the perturbative polynomials $P(x,y)$ and $Q(x,y).$ We prove:
\begin{thm}\label{thm:1.1}
Consider a system of the
form~\eqref{eq:1},
\[
\left\{
\begin{array}{rcr}
 \dot{x} &=& -yG(x,y) + \eps \,P(x,y), \\ [1.2ex]
 \dot{y} &=&  xG(x,y) + \eps \,Q(x,y), \\
\end{array}
\right.
\]
where
\[
 G(x,y)=\prod_{j=1}^{K_1} (x-\a_j) \prod_{{\ell}=1}^{K_2} (y-\b_{\ell}),
\]
$P(x,y),$ $Q(x,y)$ are polynomials of degree $n,$ $\a_j$ and $\b_\ell$ are
real numbers with $\a_i\ne \a_j $ and $\b_i\ne \b_j$ for $i\ne j,$
  $\eps$ is a small parameter and $K_1\ge K_2\ge 0.$ Moreover, when
  $K_2=0,$  $K_1\ge1$ and $\prod_{{\ell}=1}^{0} (y-\b_{\ell}):=1.$

  Let $I(r)$ be its associated Abelian integral defined in
  \eqref{eq:4}.
Then, the number of real zeroes of $I(r)$ in
$(0,\min\{|\a_j|,|\b_{\ell}|\}),$  counting their multiplicities,
$\mathcal {Z}(I),$ satisfies
\[
\mathcal{Z}(I)\leq
\begin{cases}
\widetilde{K}_1 \left[ \dfrac{n+3}{2} \right] + \left[ \dfrac{n}{2} \right],
&K_2=0,\\[0.2cm]
(\widetilde{K}_1+\widetilde{K}_2)\left( \left[ \dfrac{n+2}{2}
\right]+L\right)+\left[\dfrac{n-1}{2} \right] +L,&K_2\ge1,
\end{cases}\]
where
\begin{align*}
\widetilde{K}_1&=\card\{|\a_1|,\ldots,|\a_{K_1}|\}\le K_1,\\
\widetilde{K}_2&=\card\{|\a_1|,\ldots,|\a_{K_1}|,|\b_1|,\ldots,|\b_{K_2}|\}
-\widetilde{K}_1\le K_2,\\
L&=\card\{\a_j^2+\b_\ell^2,\,{j=1,\ldots,K_1,\,
\ell=1,\ldots,K_2}  \}\le K_1K_2
\end{align*}
and $[s]$ denotes the integer part of $s.$
\end{thm}

Note that the symmetric situation, $K_2\ge K_1,$ evolves in a
completely similar way changing $(x,y)$ by $(y,x).$

Recall that using this approach we know that the total number 
of limit cycles (counting their multiplicities) of
system~\eqref{eq:1} which bifurcate from its periodic
orbits  is bounded by the maximum number of isolated zeroes
(counting their multiplicities) of $I(r)$ for
$0<r<\rho=\min\{|\a_j|,|\b_{\ell}|\}.$

The proof of Theorem~\ref{thm:1.1} is  based on
two steps: a first one where the corresponding Abelian integral
\eqref{eq:4} is explicitly computed and a second
one where an upper estimate on the number of its zeroes is
provided. This is done in Section~\ref{se:2} and \ref{se:4},
respectively. This second part is supported on the following result,
proved in Section~\ref{se:3}, that we believe it is interesting by
itself.
\begin{thm}\label{thm:1.2}
Consider a function of the form
\begin{equation}\label{eq:5}
F(x) = P^0(x)+\sum\limits_{j=1}^KP^j(x)\frac{1}{\sqrt{x+c_j}},
\end{equation}
where  $P^j(x),$ $j=0,\ldots,K,$ are real polynomials and $c_j,
j=1,\ldots,K,$ are real constants. Then its number of real zeroes,
taking into account their multiplicities, $\mathcal{Z}(F),$
satisfies
\begin{equation}\label{eq:6}
\mathcal{Z}(F) \leq
K\left(\max_{j=1,\ldots,K}\left(\deg(P^j)\right)+1\right)+\deg(P^0).
\end{equation}
Here $\deg(0)=-1.$
\end{thm}

In the forthcoming paper \cite{GasLazTor2010a} the above result is
extended to a wider family of functions, studying in particular its
sharpness and its relation with the theory of Chebyshev systems. We
only comment here that when  $\deg(P^j)$ coincide for all
$j=1,\ldots, K$ it can be seen that the result is sharp.

In Section~\ref{se:5} we apply Theorem~\ref{thm:1.1}
to some particular cases, already studied by other authors, all satisfying  $K_1+K_2\le4.$ More concretely, in that cases we show that our theorem either gives new proofs or improve the known results for the upper bounds for the number of zeroes of  $I(r).$

The main differences between our work and the previous ones are:

\begin{itemize}

\item[-] We  manage to study the case of having an arbitrary number
$ K$ of lines of critical points for the unperturbed system. All
previous results consider at most 4 lines with several relative
positions.

\item[-] We prove a result (Theorem~\ref{thm:1.2})
for bounding the number of zeroes of a special type of functions
which are precisely the ones that appear in the final expression of
the Abelian integral $I(r).$ The previous works apply general
methods like, squaring the equation to eliminate radicals, or the
principle of the argument, extending the function to $\C.$ It can be
seen that when there are more than two square roots in the
expression~\eqref{eq:5}, Theorem~\ref{thm:1.2} is sharper
than those general methods, see \cite{GasLazTor2010a}. This is
the reason for which we can improve some of the previously given
upper bounds for $\mathcal{Z}(I).$
\end{itemize}

Some comments about the sharpness of the upper bounds provided by 
Theorem~\ref{thm:1.1} are given at the end of Section~\ref{se:4}.

\section{Explicit computation of the Abelian integral}\label{se:2}

The aim of this section is to obtain an explicit expression of the
Abelian integral presented in the introduction. The first two lemmas
deal with the cases of one and two perpendicular singular lines.
They are already known, see
\cite{BuiLli2007,LliPerRod2001}, but the proof that we
present is shorter. The next two results extend them to the case of
an arbitrary number of parallel or perpendicular lines.

\begin{lem} \label{le:2.1}
Let $\a$ be a non-zero real number. For any
$0<r<|\a|$ and any polynomial $R_{n+1}(x,y)$ of degree $n+1,$  define
\begin{equation*}
I_{n+1}^\a(r) = \int_{0}^{2\pi}
\dfrac{R_{n+1}(r\cos\theta,r\sin\theta)}{r\cos\theta-\a}\,
d\theta.
\end{equation*}
Then, for $n\geq 0,$ one has
\begin{equation*}
I_{n+1}^\a(r) =
\frac{S_{[(n-1)/2]+1}(r^2)}{\sqrt{\a^2-r^2}}+T_{[n/2]}(r^2),
\end{equation*}
for suitable polynomials $S_s(\rho)$ and $T_s(\rho)$ of degree $s.$
Moreover $I_{0}^\a(r) = -2\pi R_{0}/\sqrt{\a^2-r^2}.$
\end{lem}

\begin{proof}
It is easy to check that
\[
\int_0^{2\pi}\frac{1}{r\cos\theta-\a}d\theta=
-2\pi\frac{1}{\sqrt{\a^2-r^2}}.
\]
Hence the expression for $I_0^\a(r)$ follows.

Let us now deal with $n\ge0.$ We will proceed inductively. When
$n=0$ we have that
\[
I_1^\a(r)=\int_0^{2\pi}\frac{a_{0,0}+a_{0,1}r\cos\theta+a_{1,0}r\sin\theta}
{r\cos\theta-\a}d\theta.
\]
For this case we can write
\begin{align*}
I_1^{\a}(r)=&I_0^\a(r)+a_{0,1}\int_0^{2\pi}\frac{r\cos\theta}{r\cos\theta-\a}d\theta
+a_{1,0}\int_0^{2\pi}\frac{r\sin\theta}{r\cos\theta-\a}d\theta=\\
&I_0^\a(r)+a_{0,1}\int_0^{2\pi}\frac{r\cos\theta-\a+\a}{r\cos\theta-\a}d\theta=
I_0^\a(r)+2\pi
a_{0,1}+a_{0,1}\a
\int_0^{2\pi}\frac{1}{r\cos\theta-\a}d\theta=\\
&2\pi a_{0,1}-2\pi (a_{0,0}+a_{0,1} \a) \frac{1}{\sqrt{\a^2-r^2}}.
\end{align*}
Now we will prove the expression for $I_{n+1}^\a(r)$ by induction on
the degree. By hypothesis of induction it is enough to prove the
formula when $R_{n+1}(x,y)$ is a homogeneous polynomial of degree
$n+1.$ If we write
\[
\int_{0}^{2\pi}\frac{\sum\limits_{i=0}^{n+1}r^{n+1}a_{i,n+1-i}\sin^i\theta\cos^{n+1-i}
\theta}{r\cos\theta-\a}
d\theta=r^{n+1}\sum\limits_{i=0}^{n+1}\int_{0}^{2\pi}\frac{a_{i,n+1-i}\sin^i\theta
\cos^{n+1-i}\theta}{r\cos\theta-\a}d\theta,
\]
using symmetry properties of the integrated functions, we know that
all the integrals with an odd exponent in $\sin\theta$ are zero.
Then we have
\begin{align*}
& r^{n+1}\sum\limits_{j=0}^{[(n+1)/2]}a_{2j,n+1-2j}\int_{0}^{2\pi}\frac{\sin^{2j}
\theta\cos^{n+1-2j}\theta}{r\cos\theta-\a}d\theta =\\
& r^{n+1}\sum\limits_{j=0}^{[(n+1)/2]}a_{2j,n+1-2j}\int_{0}^{2\pi}\frac{(1-\cos^{2}
\theta)^j\cos^{n+1-2j}\theta}{r\cos\theta-\a}d\theta=\\
&
\sum\limits_{j=0}^{[(n+1)/2]}a_{2j,n+1-2j}\sum\limits_{k=0}^jr^{2j-2k}(-1)^j
\binom{j}{k}\int_{0}^{2\pi}\frac{r^{n+1+2k-2j}\cos^{n+1-2j+2k}\theta}
{r\cos\theta-\a}d\theta.
\end{align*}
When $k<j$ we can use again the induction hypotheses for each term
of the sum and the statement is proved because when we multiply by
the monomial $(r^2)^{j-k},$ the degrees, in $r^2,$ of the
polynomials $S(r^2)$ and $T(r^2)$ are $([(n+2k-2j-1)/2]+1)+(j-k)=[(n-1)/2]+1$
and $[(n+2k-2j)/2]+(j-k)=[n/2],$ respectively.

The unique case that remains to check is $k=j.$ For it we can
write
\begin{align*}
\int_{0}^{2\pi}\frac{r^{n+1}\cos^{n+1}\theta}{r\cos\theta-\a}d\theta=&
\int_{0}^{2\pi}\frac{r^{n}\cos^{n}\theta(r\cos\theta-\a+\a)}{r\cos\theta-\a}d\theta=\\
&\int_{0}^{2\pi}r^{n}\cos^{n}\theta  d\theta+
\a\int_{0}^{2\pi}\frac{r^{n}\cos^{n}\theta }{r\cos\theta-\a}d\theta.
\end{align*}
We will treat separately the cases $n$ even and $n$ odd.

When $n$ is even the first term of the above sum is a polynomial of
degree $n/2$ in $r^2$ and for the second term, using the induction
hypothesis, the polynomials $S(r^2)$ and $T(r^2)$ are of degree
$[((n-1)-1)/2]+1=(n-2)/2+1=n/2=[(n-1)/2]+1$ and $[(n-1)/2]=(n-2)/2,$
respectively. The statement is proved, in this case, because we
should add the monomial of degree $n/2$ corresponding to the first
summand. Hence the  degree of the polynomial $T(r^2)$ is $n/2=[n/2].$

When $n$ is odd the first term in the sum is zero. For the second
term, from the induction hypothesis, the polynomials $S(r^2)$ and $T(r^2)$ are
of degree $[((n-1)-1)/2]+1=(n-3)/2+1=(n-1)/2=[(n-1)/2]$ and
$[(n-1)/2]=[n/2],$ respectively. Note that the contribution of this
term to the polynomial $S(r^2)$ has one degree less than expected but the
total degree is $[(n-1)/2]+1$ because it appears for the other
terms. More concretely in the $(\sin\theta)^{[(n+1)/2]}$ term.
\end{proof}

\begin{lem} \label{le:2.2}
Let $\a$ and $\b$ be non-zero real numbers. For any
$0<r<\min(|\a|,|\b|)$ and any polynomial $R_{n+1}(x,y)$ of degree $n+1$ consider
\begin{equation*}
I_{n+1}^{\a,\b}(r)=\int_{0}^{2\pi}
\dfrac{R_{n+1}(r\cos\theta,r\sin\theta)}{(r\cos\theta-\a)(r\sin\theta-\b)}\,
d\theta.
\end{equation*}
Then, for $n\geq 0,$ we have that
\begin{equation*}
I_{n+1}^{\a,\b}(r)=\frac{1}{\a^2+\b^2-r^2}
\left(\frac{U_{[n/2]+1}(r^2)}{\sqrt{\a^2-r^2}}+\frac{V_{[n/2]+1}(r^2)}
{\sqrt{\b^2-r^2}}\right)
+W_{[(n-1)/2]}(r^2),
\end{equation*}
for some given polynomials $U_s(\rho),$ $V_s(\rho)$ and $W_s(\rho)$
of degree $s.$ Moreover
\begin{equation*}
I_{0}^{\a,\b}(r)=2\pi\frac{R_0}{\a^2+\b^2-r^2}
\left(\frac{\a}{\sqrt{\a^2-r^2}}+\frac{\b}{\sqrt{\b^2-r^2}}\right).
\end{equation*}
\end{lem}

\begin{proof}
Applying for instance residues formula, we obtain
\[
\int_0^{2\pi}\frac{d\theta}{(r\cos\theta-\a)(r\sin\theta-\b)}=
\frac{2\pi}{\a^2+\b^2-r^2}
\left(\frac{\a}{\sqrt{\a^2-r^2}}+\frac{\b}{\sqrt{\b^2-r^2}}\right).
\]
Therefore the expression for $I_0^{\a,\b}(r)$ follows.

To obtain the result for all $n$ we will proceed inductively.
When $n=0$ we have that
\[
I_1^{\a,\b}(r)=\int_0^{2\pi}\frac{a_{0,0}+a_{0,1}r\cos\theta+a_{1,0}r\sin\theta}
{(r\cos\theta-\a)(r\sin\theta-\b)}\,d\theta.
\]
For this case we can write
\begin{align*}
I_1^{\a,\b}(r)=&I_0^{\a,\b}(r)+a_{0,1}\int_0^{2\pi}\frac{r\cos\theta-\a+\a}
{(r\cos\theta-\a)(r\sin\theta-\b)}\,d\theta+\\
&a_{1,0}\int_0^{2\pi}
\frac{r\sin\theta-\b+\b}{(r\cos\theta-\a)(r\sin\theta-\b)}d\theta=I_0^{\a,\b}(r)+
a_{0,1}\int_0^{2\pi}\frac{1}{r\sin\theta-\b}\,d\theta+\\&a_{1,0}\int_0^{2\pi}
\frac{1}{r\cos\theta-\a}\,d\theta+
(a_{0,1}\a+a_{1,0}\b)\int_0^{2\pi}\frac{1}{(r\cos\theta-\a)(r\sin\theta-\b)}\,d\theta=\\
& a_{0,1} \frac{-2\pi}{\sqrt{\b^2-r^2}}+ a_{1,0}
\frac{-2\pi}{\sqrt{\a^2-r^2}}+
\frac{a_{0,0}+a_{0,1}\a+a_{1,0}\b}{\a^2+\b^2-r^2}\left(\frac{2\pi
\a}{\sqrt{\b^2-r^2}}+\frac{2\pi \b}{\sqrt{\a^2-r^2}}\right).
\end{align*}
This last expression satisfies the statement with $U(r^2)$ and $V(r^2)$
polynomials of degree $[n/2]+1=[0/2]+1=1$ in $r^2$ and the
polynomial $W(r^2)$ is identically zero.

Now we will prove the expression for $I_{n+1}^{\a,\b}(r)$ by induction
on the degree of $R_n(x,y).$ As in the proof of Lemma~\ref{le:2.1}, by
induction hypothesis, it is enough to prove the result  when $R_{n+1}(x,y)$ is a
homogeneous polynomial of degree $n+1.$ For this case we can write
\begin{align*}
&\int_{0}^{2\pi}\frac{\sum\limits_{i=0}^{n+1}r^{n+1}a_{i,n+1-i}\sin^i\theta\cos^{n+1-i}
\theta}{(r\cos\theta-\a)(r\sin\theta-\b)}\,
d\theta=\sum\limits_{i=0}^{n+1}\int_{0}^{2\pi}\frac{r^{n+1} a_{i,n+1-i}\sin^i\theta
\cos^{n+1-i}\theta}{(r\cos\theta-\a)(r\sin\theta-\b)}\,d\theta=\\
&\sum\limits_{j=0}^{\left[\frac{n}{2}\right]}\int_{0}^{2\pi}\frac{r^{n+1}
a_{2j+1,n+1-(2j+1)}\sin\theta(1-\cos^2\theta)^j\cos^{n+1-(2j+1)}\theta}
{(r\cos\theta-\a)(r\sin\theta-\b)}\,d\theta+\\
&\sum\limits_{j=0}^{\left[\frac{n+1}{2}\right]}\int_{0}^{2\pi}
\frac{r^{n+1}a_{2j,n+1-2j}(1-\cos^2\theta)^j\cos^{n+1-2j}\theta}
{(r\cos\theta-\a)(r\sin\theta-\b)}\,d\theta=\\
&\sum\limits_{j=0}^{\left[\frac{n}{2}\right]}
\sum_{k=0}^{j}\binom{j}{k}(-1)^k
a_{2j+1,n-2j}r^{2(j-k)}\int_{0}^{2\pi}\frac{r^{n-2(j-k)+1}\sin\theta\cos^{n-2(j-k)}\theta}
{(r\cos\theta-\a)(r\sin\theta-\b)}\,d\theta+\\
&\sum\limits_{j=0}^{\left[\frac{n+1}{2}\right]}
\sum_{k=0}^{j}\binom{j}{k}(-1)^k
a_{2j,n+1-2j}r^{2(j-k)}\int_{0}^{2\pi}\frac{r^{n+1-2(j-k)}\cos^{n+1-2(j-k)}\theta}
{(r\cos\theta-\a)(r\sin\theta-\b)}\,d\theta.
\end{align*}
When $k<j$ the statement is proved using the induction hypothesis in
each of the terms of the sum since, after the multiplication by
the monomial $(r^2)^{j-k},$ we have that the polynomials $U(r^2),$ $V(r^2)$
and $W(r^2)$ have degree $[(n-2(j-k))/2]+1+(j-k)=[n/2]+1,$
$[(n-2(j-k))/2]+1+(j-k)=[n/2]+1$ and
$[(n-2(j-k)-1)/2]+(j-k)=[(n-1)/2],$ respectively.

The unique case that remains to check is $k=j$ for both integrals:
\[
\int_{0}^{2\pi}\frac{r^{n+1}\sin\theta\cos^n\theta}{(r\cos\theta-\a)(r\sin\theta-\b)}
\,d\theta
\quad \textrm{and}\quad
\int_{0}^{2\pi}\frac{r^{n+1}\cos^{n+1}\theta}{(r\cos\theta-\a)(r\sin\theta-\b)}\,d\theta.
\]
For the first one we can write
\begin{align*}
&\int_{0}^{2\pi}\frac{r^{n+1}\sin\theta\cos^n\theta}{(r\cos\theta-\a)(r\sin\theta-\b)}
\,d\theta
=
\int_{0}^{2\pi}\frac{r^{n}(r\sin\theta-\b+\b)\cos^n\theta}{(r\cos\theta-\a)
(r\sin\theta-\b)}d\theta=\\&\int_{0}^{2\pi}\frac{r^{n}\cos^n\theta}
{r\cos\theta-\a}\,d\theta
+\b\int_{0}^{2\pi}\frac{r^{n}\cos^n\theta}{(r\cos\theta-\a)(r\sin\theta-\b)}\,d\theta.
\end{align*}
Thus we can take only the first term in the sum because for the
second one we can apply once more the induction hypothesis. For this
term, using the computation of $I_{n}^{\a}(r),$ the corresponding
polynomial $W(r^2),$ that is called $T$ in Lemma~\ref{le:2.1}, has degree
$[(n-1)/2]$ and the polynomial $U(r^2)$ can be written as
$(\a^2+\b^2-r^2)S_{[(n-2)/2]+1}(r^2)$ and it has degree
$([(n-2)/2]+1)+1=[n/2]+1.$

Now we will consider the last integral, that is
\begin{align*}
&\int_{0}^{2\pi}\frac{r^{n+1}\cos^{n+1}\theta}{(r\cos\theta-\a)(r\sin\theta-\b)}d\theta
=
\int_{0}^{2\pi}\frac{r^n\cos^n\theta(r\cos\theta-\a+\a)}{(r\cos\theta-\a)
(r\sin\theta-\b)}d\theta =\\
&\int_{0}^{2\pi}\frac{r^n\cos^n\theta}{r\sin\theta-\b}d\theta
+\a\int_{0}^{2\pi}\frac{r^n\cos^n\theta}{(r\cos\theta-\a)(r\sin\theta-\b)}d\theta.
\end{align*}
Changing $\theta$ by $\theta+\pi/2$ we can use the expression for
$I^{\b}_{n}(r)$ computed in Lemma~\ref{le:2.1} to prove that, as in the
previous computation, this integral satisfies the formula of the
statement. The second term follows applying the induction
hypothesis. Then the statement is proved.
\end{proof}

\begin{prop} \label{pr:2.3}
Let $K$ be a natural number, let $P_n(x,y)$
and $Q_n(x,y)$ be real polynomials of degree $n,$ $\gamma_r= \left\{
(x,y)\ \vert \ x^2+y^2=r^2\right\},$ $\{\a_1,\ldots,\a_{K}\}$
different real numbers and
\[
I(r)= \int_{\gamma_r} \dfrac{Q_n(x,y) \, dx - P_n(x,y) \,
dy}{\prod\limits_{j=1}^{K} (x-\a_j) }.
\]
Then,
\begin{equation}\label{eq:7}
I(r)=
\sum\limits_{j=1}^{\widetilde{K}}\frac{S^j_{[(n-1)/2]+1}(r^2)}
{\sqrt{\widetilde{\a}_j^2-r^2}}+T_{[n/2]}(r^2),
\end{equation}
for suitable polynomials $S^j_s(\rho),$ $T_s(\rho)$ of degree $s,$
where $\widetilde{K}=\card\{|\a_1|,\ldots,|\a_K|\}$ and
$\widetilde{\a}_1,\ldots,\widetilde{\a}_{\widetilde{K}}$ denoting
the different values of the set $\{|\a_1|,\ldots,|\a_K|\}.$
\end{prop}

\begin{proof}
Parametrising $\gamma_r$ using polar coordinates,
$(x,y)=(r\cos\theta,r\sin\theta),$ we can write
\[
I(r)=\int_0^{2\pi}\frac{R_{n+1}(r\cos\theta,r\sin\theta)}{\prod\limits_{j=1}^{K}
(r\cos\theta-\a_j)}\, d\theta
\]
where $R_{n+1}(r\cos\theta,r\sin\theta)= - r
Q_n(r\cos\theta,r\sin\theta)\sin\theta - r
P_n(r\cos\theta,r\sin\theta)\cos\theta.$ Performing partial fraction
decomposition and using Lemma~\ref{le:2.1}, it turns out that
\begin{align*}
I(r) =& \sum\limits_{j=1}^{K}\int_0^{2\pi}\frac{1}{\prod
\limits_{\fracz{\ell=1}{\ell\ne j}}^K(\a_{\ell}-\a_j)}
\frac{R_{n+1}(r\cos\theta,r\sin\theta) }{r\cos\theta-\a_j}\,d\theta= \\[1.2ex]
&
\sum\limits_{j=1}^{K}\left(\frac{S^j_{[(n-1)/2]+1}(r^2)}
{\sqrt{\a_j^2-r^2}}+T^j_{[n/2]}(r^2)\right)
=
\sum\limits_{j=1}^{K}\left(\frac{S^j_{[(n-1)/2]+1}(r^2)}{\sqrt{\a_j^2-r^2}}
\right) + T_{[n/2]}(r^2),
\end{align*}
provided we define $T_{[n/2]}(\rho)=\sum_{j=1}^K T^j_{[n/2]}(\rho).$
Since this expression depends only on the absolute values $|\a_j|$
we collect terms, consider new polynomials $S^j_s(\rho)$ and, at the end,
get formula~\eqref{eq:7}.
\end{proof}

\begin{prop} \label{pr:2.4}
Let $K_1,$ $K_2$ be natural
numbers, $P_n(x,y)$ and $Q_n(x,y)$ be real polynomials of degree
$n,$ $\gamma_r= \left\{ (x,y)\ \vert \ x^2+y^2=r^2\right\}$ and
$\{\a_1,\ldots,\a_{K_1},\b_1,\ldots,\b_{K_2}\}$ real numbers
satisfying that $\a_j\ne \a_{\ell}$ and $\b_j\ne \b_{\ell}$  for
$j\ne\ell.$ Consider
\[
I(r)= \int_{\gamma_r} \dfrac{Q_n(x,y) \, dx - P_n(x,y) \,
dy}{\prod\limits_{j=1}^{K_1} (x-\a_j) \prod\limits_{k=1}^{K_2}
(y-\b_k)}.
\]
Then,
\begin{multline*}
I(r)= \sum\limits_{j=1}^{K_1}\left(\sum\limits_{k=1}^{K_2}
\frac{U^{j,k}_{[n/2]+1}(r^2)}{\a_j^2+\b_k^2-r^2}\right)\frac{1}{\sqrt{\a_j^2-r^2}}+\\
\sum\limits_{k=1}^{K_2}\left(\sum\limits_{j=1}^{K_1}
\frac{V^{j,k}_{[n/2]+1}(r^2)}{\a_j^2+\b_k^2-r^2}\right)
\frac{1}{\sqrt{\b_k^2-r^2}}+W_{[(n-1)/2]}(r^2),
\end{multline*}
for suitable polynomials  $U^{j,k}_s(\rho),$ $V^{j,k}_s(\rho)$ and
$W_s(\rho)$ of degree $s.$
\end{prop}

\begin{proof}
Parameterising $\gamma_r$ using polar coordinates,
$(x,y)=(r\cos\theta,r\sin\theta),$ we can write
\[
I(r)=\int_0^{2\pi}\frac{R_{n+1}(r\cos\theta,r\sin\theta)
}{\prod\limits_{j=1}^{K_1} (r\cos\theta-\a_j)
\prod\limits_{k=1}^{K_2} (r\sin\theta-\b_k)}\,d\theta,
\]
where $R_{n+1}(r\cos\theta,r\sin\theta)= - r
Q_n(r\cos\theta,r\sin\theta)\sin\theta - r
P_n(r\cos\theta,r\sin\theta)\cos\theta.$ Performing a partial
fraction expansion and using Lemma~\ref{le:2.2} it follows that
\begin{align*}
I(r)=& \sum\limits_{j=1}^{K_1}\sum\limits_{k=1}^{K_2}\int_0^{2\pi}
\frac{1}{\prod\limits_{\fracz{{\ell}=1}{{\ell}\ne j}}^{K_1}
(\a_{\ell}-\a_j)}\frac{1}{\prod\limits_{\fracz{{\ell}=1}
{{\ell}\ne k}}^{K_2}(\b_{\ell}-\b_k)}\frac{R_{n+1}
(r\cos\theta,r\sin\theta)}{(r\cos\theta-\a_j)(r\sin\theta-\b_k)}\,d\theta=\\
& \sum\limits_{j=1}^{K_1}\sum\limits_{k=1}^{K_2}\left(
\frac{1}{\a_j^2+\b_k^2-r^2}\left( \frac{U^{j,k}_{[n/2]+1}(r^2)}
{\sqrt{\a_j^2-r^2}}+\frac{V^{j,k}_{[n/2]+1}(r^2)}{\sqrt{\b_k^2-r^2}}\right)
 +W^{j,k}_{[(n-1)/2]}(r^2)\right)=\\
& \sum\limits_{j=1}^{K_1}\left(\sum\limits_{k=1}^{K_2}
\frac{U^{j,k}_{[n/2]+1}(r^2)}{\a_j^2+\b_k^2-r^2}\right)\frac{1}{\sqrt{\a_j^2-r^2}}
+\sum\limits_{k=1}^{K_2}\left(\sum\limits_{j=1}^{K_1}
\frac{V^{j,k}_{[n/2]+1}(r^2)}{\a_j^2+\b_k^2-r^2}\right) \frac{1}{\sqrt{\b_k^2-r^2}}+\\
&\sum\limits_{j=1}^{K_1}\sum\limits_{k=1}^{K_2}
W^{j,k}_{[(n-1)/2]}(r^2),
\end{align*}
for a collection of polynomials $U^{j,k}_s(\rho),$
$V^{j,k}_s(\rho),$ $W^{j,k}_s(\rho)$ of degree $s$ for all
$j=1,\ldots,K_1$ and $k=1,\ldots,K_2.$ Denoting
$W_{[(n-1)/2]}(r^2)=\sum\limits_{j=1}^{K_1}\sum\limits_{k=1}^{K_2}
W^{j,k}_{[(n-1)/2]}(r^2),$ the claimed result follows.
\end{proof}

\section{Proof of Theorem \ref{thm:1.2}}\label{se:3}

The proof of  Theorem \ref{thm:1.2} will use the known as
Derivation-Division procedure (see for instance \cite[p. 119]{Rou1998}).
First  we give two technical results useful to compute successive
de\-ri\-va\-ti\-ves that appear applying this procedure to the function
\eqref{eq:5}. The proof of the first one is very simple and implies that, in general,
\begin{equation} \label{eq:8}
\mathcal{D}^{k}
\left(p_n(x)\left(\frac{x+a}{x+b}\right)^{\alpha}\right) = 
q_{n+k}(x)\frac{(x+a)^{\alpha-k}}{(x+b)^{\alpha+k}}.
\end{equation}
The proof of the second one is a little more involved and, as we will see, it
constitutes the key point for proving Theorem \ref{thm:1.2}.
It shows that the case 
$k=n+1$ in \eqref{eq:8} is very special, since it undergoes a sudden drop of
the degree of the polynomial $q_{n+k}(x).$

\begin{lem}\label{le:3.1}
Consider $a\neq b$ and $\alpha\not\in\Z$  real numbers. Then
\begin{equation*}
\mathcal{D} \left(p_n(x)(x+a)^{\alpha}\right) =
q_n(x)(x+a)^{\alpha-1}, \quad \mathcal{D}
\left(p_n(x)\left(\frac{x+a}{x+b}\right)^{\alpha}\right) =
q_{n+1}(x)\frac{(x+a)^{\alpha-1}}{(x+b)^{\alpha+1}},
\end{equation*}
where $p_i(x)$ and $q_i(x)$ are polynomials of degree $i.$
\end{lem}

\begin{lem}\label{le:3.2}
Consider $a\neq b$ and  $\alpha\not\in\Z$   real numbers. Then
\begin{equation*}
\mathcal{D}^{n+1}
\left(p_n(x)\left(\frac{x+a}{x+b}\right)^{\alpha}\right) =
q_{n}(x)\frac{(x+a)^{\alpha-(n+1)}}{(x+b)^{\alpha+(n+1)}},
\end{equation*}
where $p_n(x)$ and $q_n(x)$ are polynomials of degree $n.$
\end{lem}

\begin{proof}
We will prove the statement inductively. It is not restrictive to
consider $p_n(x)$ as a monic polynomial and written as
$p_n(x)=(x+a)^n+p_{n-1}(x)$ for a suitable polynomial $p_{n-1}(x)$
of degree $n-1.$ From the equalities
\begin{equation*}
\begin{aligned}
&\mathcal{D}^{n+1}
\left(p_n(x)\left(\frac{x+a}{x+b}\right)^{\alpha}\right) =
\mathcal{D}^{n+1} \left(\left((x+a)^n+p_{n-1}(x)\right)\left(\frac{x+a}{x+b}
\right)^{\alpha}\right) =\\
&\mathcal{D}^{n+1}
\left((x+a)^n\left(\frac{x+a}{x+b}\right)^{\alpha}\right)
+\mathcal{D}\left(\mathcal{D}^{n}
\left(p_{n-1}(x)\left(\frac{x+a}{x+b}\right)^{\alpha}\right)\right)
\end{aligned}
\end{equation*}
and using Lemma~\ref{le:3.1}, we have that it is enough to
prove the statement for $p_n(x)=(x+a)^{n}.$

Indeed, it is not difficult to check, inductively, that the
following expression is satisfied
\begin{equation*}
\mathcal{D}^{j}
\left(\frac{(x+a)^{\alpha+n}}{(x+b)^{\alpha}}\right)=\frac{(x+a)^{\alpha+n-j}}
{(x+b)^{\alpha+j}}\sum\limits_{\ell=0}^{j}C_{j,\ell}(x+a)^{\ell}
\end{equation*}
with
\[C_{j,\ell}=
(-1)^{j+\ell}(a-b)^{j-\ell}\binom{j}{\ell}\frac{\Gamma(\alpha+n+1)}
{\Gamma(\alpha+n-(j-\ell)+1)}\prod\limits_{m=1}^{\ell}(n-j+m),
\]
where we define $\prod\limits_{m=1}^{0}(n-j+m)=1$ and  $\Gamma$ is
the Gamma function. Therefore,
\begin{multline*}
\mathcal{D}^{n+1} \left((x+a)^n
\left(\frac{x+a}{x+b}\right)^{\alpha}\right)=
\frac{(x+a)^{\alpha+n-(n+1)}}{(x+b)^{\alpha+(n+1)}}
\sum\limits_{\ell=0}^{(n+1)}C_{n+1,\ell}(x+a)^{\ell}=\\
\frac{(x+a)^{\alpha-1}}{(x+b)^{\alpha+n+1}}\Bigg(
(-1)^{n+1}(a-b)^{n+1}\frac{\Gamma(\alpha+n+1)}{\Gamma(\alpha)}+\\
\sum\limits_{\ell=1}^{n+1}(-1)^{n+1+\ell}(a-b)^{n+1-\ell}
\binom{n+1}{\ell}\frac{\Gamma(\alpha+n+1)}{\Gamma(\alpha+\ell)}
\prod\limits_{m=1}^{\ell}(m-1)\Bigg).
\end{multline*}

In the latter expression, the terms corresponding to $\ell=1,\ldots,n+1,$ clearly vanish and then
it reduces to
\[
(-1)^{n+1}(a-b)^{n+1}\frac{\Gamma(\alpha+n+1)}{\Gamma(\alpha)}(x+a)^{n}
\frac{(x+a)^{\alpha-(n+1)}}{(x+b)^{\alpha+n+1}}.
\]
Hence the statement is proved.

\end{proof}

\begin{proof}[Proof of Theorem~\ref{thm:1.2}] Set $n_0=\deg(P^0).$
Differentiating $F(x)$ in \eqref{eq:5} $n_0+1$ times, applying
Lemma~\ref{le:3.1} and dividing by
$(x+c_1)^{-\frac{1}{2}-(n_0+1)}$ we have
\begin{equation*}
F_1(x)=P_{n}^{1,1}(x)+ \sum\limits_{j=2}^K P_{n}^{j,1}(x)
\left(\frac{x+c_j}{x+c_1}\right)^{\alpha_1},
\end{equation*}
where $P_{n}^{j,1}(x)$ are suitable polynomials of degree at most
$n=\max(\deg{P^1},\ldots, \deg{P^K})$ and
$\alpha_1=-\frac{1}{2}-(n_0+1).$ Differentiating $F_1(x),$ $n+1$ times,
applying Lemma~\ref{le:3.2} and dividing by
$(x+c_2)^{\alpha_1-(n+1)}/(x+c_1)^{\alpha_1+(n+1)}$ we have
\begin{equation*}
F_2(x)=P_{n}^{2,2}(x)+ \sum\limits_{j=3}^K P_{n}^{j,2}(x)
\left(\frac{x+c_j}{x+c_2}\right)^{\alpha_2},
\end{equation*}
with $P_{n}^{2,j}(x)$ polynomials of degree $n$ and
$\alpha_2=\alpha_1-(n+1).$  Performing the same procedure $K-2$
times we reach
\begin{equation*}
F_K(x)=P_n^{K,K}(x).
\end{equation*}
Note that all the polynomials $P_n^{j,\ell}(x)$ appearing in the
process have degree $n.$ Since the total number of derivatives is
$(K-1)(n+1)+n_0+1$ and the degree of the last polynomial is $n$ the
total number of zeroes of $F(x)$ is bounded by $K(n+1)+n_0,$ as we
wanted to prove.
\end{proof}

\section{Proof of Theorem~\ref{thm:1.1}}\label{se:4}

We start studying the case $K_2=0.$ From the expression of
Proposition~\ref{pr:2.3} and applying
Theorem~\ref{thm:1.2} we know that the number of zeroes of
$I(r)$ is less or equal than
\[
\widetilde{K}_1 \left(\left[\frac{n-1}{2} \right]+2 \right) + \left[
\frac{n}{2} \right]=\widetilde{K}_1 \left[ \frac{n+3}{2} \right]+
\left[ \frac{n}{2} \right].
\]
as we wanted to see.

Consider now $K_2\ge1.$ We know that the set
$\{\a_j^2+\b_\ell^2,\,{j=1,\ldots,K_1,\, \ell=1,\ldots,K_2}\}$
has $L$ elements. Define the function
\[
H(r^2):=\prod_{i=1}^{L} \left(\a_{j_i}^2+\b_{\ell_i}^2-r^2\right),
\]
using all the different elements of the set.  From
Proposition~\ref{pr:2.4} and multiplying $I(r)$ by
$H(r^2)$ we obtain that
\begin{align*}
H(r^2) I(r)= &\sum\limits_{j=1}^{K_1}\left(\sum\limits_{k=1}^{K_2}
\frac{H(r^2)}{\a_j^2+\b_k^2-r^2}U^{j,k}_{[n/2]+1}(r^2)\right)\frac{1}
{\sqrt{\a_j^2-r^2}}+\\
&\sum\limits_{k=1}^{K_2}\left(\sum\limits_{j=1}^{K_1}
\frac{H(r^2)}{\a_j^2+\b_k^2-r^2}V^{j,k}_{[n/2]+1}(r^2)\right)
\frac{1}{\sqrt{\b_k^2-r^2}}+H(r^2)W_{[(n-1)/2]}(r^2)=\\
&\sum\limits_{j=1}^{K_1}\widehat U^j_{[n/2]+L}(r^2)\frac{1}
{\sqrt{\a_j^2-r^2}}+\sum\limits_{k=1}^{K_2} \widehat
V^k_{[n/2]+L}(r^2)\frac{1}{\sqrt{\b_k^2-r^2}}+\widehat W_{[(n-1)/2]+L}(r^2)=\\
&\sum\limits_{j=1}^{\widetilde {K}_1}\widehat
S^j_{[n/2]+L}(r^2)\frac{1}{\sqrt{\widetilde{\a}_j^2-r^2}}+\sum\limits_{k=1}^{\widetilde
{K}_2} \widehat T^k_{[n/2]+L}(r^2)\frac{1}{\sqrt{\widetilde
{\b}_k^2-r^2}}+\widehat W_{[(n-1)/2]+L}(r^2),
\end{align*}
where  the polynomials $\widehat{S}_m^j(r^2),$ $\widehat{T}_m^k(r^2),$
$\widehat{U}_m^j(r^2),$ $\widehat{V}_m^k(r^2)$ and $\widehat{W}_m(r^2)$ have
degree $m$ in $r^2.$ In the last step
$\widetilde{\a}_1,\ldots,\widetilde{\a}_{\widetilde{K}_1},$ denote
the different values of the set $\{|\a_1|,\ldots,|\a_{K_1}|\}$ and
$\widetilde{\b}_1,\ldots,\widetilde{\b}_{\widetilde{K}_2},$ denote
the different values of the set
$\{|\b_1|,\ldots,|\b_{K_2}|\}\setminus
\{|\a_1|,\ldots,|\a_{K_1}|\}.$ Remember that
\[  \widetilde{K}_1=\card\{|\a_1|,\ldots,|\a_{K_1}|\} \quad
\mbox{and}\quad \widetilde{K}_2=\card\{|\a_1|,\ldots,|
\a_{K_1}|,|\b_1|,\ldots,|\b_{K_2}|\}-\widetilde{K}_1.
\]

Applying Theorem~\ref{thm:1.2} to the right hand part of the
last equality we have the expression of the statement because the
number of zeroes of $I(r)$  in $(0,\min\{|\a_j|,|\b_{\ell}|\})$
satisfies
\begin{align*}
\mathcal{Z}(I)\leq
(\widetilde{K}_1+\widetilde{K}_2)\left([n/2]+L+1\right)+[(n-1)/2]+L.
\end{align*}
\qed

\begin{rem}\label{re:4.1}
The upper bounds given in Theorem~\ref{thm:1.1}
are maximal when all the straight lines are ``generically'' located,
that is, when $\widetilde{K}_1=K_1,$  $\widetilde{K}_2=K_2$ and
$L=K_1K_2$ hold.
\end{rem}

Denote by $K$ the total number of straight lines (parallel or
orthogonal one-to-one) that we consider in our system. Then
$K=K_1+K_2.$ Therefore, the maximal value for the upper bound for
the number of zeroes of the Abelian integral $I(r)$ is achieved for
$K_1=K_2=K/2$ if $K$ even and $K_1=(K+1)/2,$ $K_2=(K-1)/2$ if $K$
odd. We have the following result:

\begin{cor}\label{co:4.2}
Under the  hypotheses of Theorem~\ref{thm:1.1} and denoting by
$K=K_1+K_2$ the total of number of straight lines of singular
points, we have that the number of zeroes of $I(r)$ satisfies
\[
\mathcal{Z}(I) \leq
K\left[\frac{n}{2}\right]+\left[\frac{n-1}{2}\right]+(K+1)
\left[\frac{K}{2}\right]\left[\frac{K+1}{2}\right]+K.
\]
\end{cor}

Although in this work we do not consider the problem of the sharpness of the
upper bounds for the Abelian integral $I(r)$ given in
Theorem~\ref{thm:1.1} we end this section with
several comments about this question. Fixed some $K_1$ and $K_2,$
first let us count how many effective free parameters gives the
perturbation $(P,Q)$ in the expression of $I(r).$ In principle
these two polynomials of degree $n$ have both together
$(n+2)(n+1)$ parameters, but  some symmetry considerations based on the first expression of $I(r),$ reduce this number in the computation of $I(r)$ to
$(n+4)(n+1)/2$ when $K_1K_2\ne 0$ and to $[(n+5)(n+1)/4]$ when
$K_1K_2=0.$ On the other hand the growth of the upper bound for
$\mathcal{Z}(I)$ is linear in $n.$ This fact makes natural to
believe that, fixed $K_1$ and $K_2$, the upper bounds for $n$ large
enough are essentially reached.

Besides, when we restrict our study to families of
polynomial vector fields  of degree $n$ it is natural to take
$K=n-1.$ Thus, applying the previous results, the upper bound for
the number of zeroes of the corresponding Abelian integral $I(r)$
 grows as $n^3/4,$ when $K_1K_2\ne0,$ while the number of free
parameters only grows as $n^2/2.$  This difference in growth
implies, for $n$ large enough, the existence of a collection of
relations between the coefficients of the functions  $P^j(x)$ that
appear when we apply Theorem~\ref{thm:1.2} to the final
expression of $I(r).$ Therefore, in this situation, it can not be
expected at all the sharpness of the upper bound.  When $K_1K_2=0$
something similar happens.

\section{Relations with some previous works}\label{se:5}

There are some previous works where the results of this paper can be
applied. In some cases we can give an alternative proof of the
corresponding upper bound; in the other cases the upper bound is
improved. We remark that in some of these works the sharpness of the
obtained upper bounds is also studied. In  our paper, this
question has not been considered.

In \cite{LliPerRod2001} the authors deal with the case of a
single straight line of critical points, $G(x,y)=x-1.$ In this case
they proved that $\mathcal{Z}(I)\le n$ and that the bound is sharp.
Our result gives an alternative proof for the upper bound.

The case with two straight lines of singularities was studied in
\cite{XiaHan2004b} and \cite{BuiLli2007}. The first one considered
the case of both lines been parallel and the latter when they are
perpendicular. In the first paper the authors proved that the number
of zeroes satisfies
\[
\mathcal{Z}(I)\leq \begin{cases}2n+3, & n \textrm{ is odd,} \\
2n+1, &n \textrm{ is even,}  \end{cases}
\]
while the upper bound provided by Theorem~\ref{thm:1.1} is sharper:
\[
\mathcal{Z}(I)\leq \begin{cases}\frac{3n+5}{2}, & n \textrm{ is
odd,}
\\  \frac{3n+4}{2}, &n \textrm{ is even.}  \end{cases}
\]
In the second paper, it was proved that
\[
\begin{cases}\frac{3n+1}{2}, & n \textrm{ is odd,} \\
\frac{3n-2}{2}, &n \textrm{ is even,}  \end{cases} \leq
\mathcal{Z}(I)\leq \begin{cases} \frac{3n+5}{2}, & n \textrm{ is
odd,} \\  \frac{3n+2}{2}, &n \textrm{ is even.}  \end{cases}
\]
Using Theorem~\ref{thm:1.1} we get that
\[
\mathcal{Z}(I)\leq \begin{cases}\frac{3n+7}{2}, & n \textrm{ is
odd,}
\\  \frac{3n+8}{2}, &n \textrm{ is even.}  \end{cases}
\]
The difference can be explained because the perturbation in
\cite{BuiLli2007} satisfies $P_{n}(0,0)=Q_{n}(0,0)=0$ while our
study has not this restriction.

In \cite{ColLliPro2010} the authors considered three lines, two
parallel and one perpendicular. They proved that
\[
\begin{cases}2n+2, & n \textrm{ is odd,} \\  2n, &n \textrm{ is even,}
\end{cases} \leq \mathcal{Z}(I)\leq \begin{cases}\frac{5n+23}{2}, & n \textrm{ is odd,}
\\  \frac{5n+18}{2}, &n \textrm{ is even.}  \end{cases}
\]
Using Theorem~\ref{thm:1.1}, with
$(K_1,K_2)=(2,1),$ the last inequalities can be improved to
\[
\mathcal{Z}(I)\leq \begin{cases}2n+12, & n \textrm{ is odd,} \\
2n+13, &n \textrm{ is even.}  \end{cases}
\]
The  special case with four lines given by
$G(x,y)=(x^2-a^2)(y^2-b^2)$ was studied in \cite{AtaNyaZan2009}. When
$a\ne b$  the authors showed that
\[
\mathcal{Z}(I)\leq \begin{cases}\frac{3n+5}{2}, & n \textrm{ is
odd,}
\\  \frac{3n+2}{2}, &n \textrm{ is even,}  \end{cases}
\]
and it was claimed that these bounds are sharp. The upper bound given by
Theorem~\ref{thm:1.1} increases in one unit the
previous ones. In our notation  $\a_1=-\a_2=a$ and $\b_1=-\b_2=b,$
and so $K_1=K_2=2$ but $\widetilde K_1=\widetilde K_2=L=1.$ In the
generic case ($K_1=K_2=\widetilde{K}_1=\widetilde{K}_2=L=2$) we get
an upper bound that grows like $5n/2.$

\end{document}